\documentclass[12pt, twoside]{article}
\usepackage{amsmath,amsthm,amssymb}
\usepackage{times}
\usepackage{enumerate}
\usepackage{extarrows}
\usepackage{color}
\usepackage{tabularx}
\usepackage{multirow}

\def\titlerunning#1{\gdef\titrun{#1}}
\makeatletter
\def\author#1{\gdef\autrun{\def\and{\unskip, }#1}\gdef\@author{#1}}
\def\address#1{{\def\and{\\\hspace*{18pt}}\renewcommand{\thefootnote}{}%
\footnote {#1}}%
\markboth{\autrun}{\titrun}}
\makeatother
\def\email#1{e-mail: #1}
\def\subjclass#1{{\renewcommand{\thefootnote}{}%
\footnote{\emph{Mathematics Subject Classification (2010):} #1}}}
\def\keywords#1{\par\medskip
\noindent\textbf{Keywords.} #1}

\newtheorem{theorem}{Theorem}[section]
\newtheorem{lemma}[theorem]{Lemma}

\newtheorem{corollary}[theorem]{Corollary}
\newtheorem{example}[theorem]{Example}

\newcommand{\R}{\mathbb{R}}

\newcommand{\Proof}{\begin{proof}}
\newcommand{\End}{\end{proof}}

\numberwithin{equation}{section}

\newcommand{\PreserveBackslash}[1]{\let\temp=\\#1\let\\=\temp}
\newcolumntype{C}[1]{>{\PreserveBackslash\centering}p{#1}}
\newcolumntype{R}[1]{>{\PreserveBackslash\raggedleft}p{#1}}
\newcolumntype{L}[1]{>{\PreserveBackslash\raggedright}p{#1}}

\newcolumntype{I}{!{\vrule width 1pt}}
\newlength\savedwidth

\begin{document}


\baselineskip=15pt


\titlerunning{Weak KAM solutions of Hamilton-Jacobi equations}

\title{Finite-time convergence of solutions of Hamilton-Jacobi equations}

\author{Kaizhi Wang\dag \and Jun Yan\ddag\and Kai Zhao*}

\date{\today}

\maketitle

\address{Kaizhi Wang: School of Mathematical Sciences, Shanghai Jiao Tong University, Shanghai 200240, China; \email{kzwang@sjtu.edu.cn}
\and Jun Yan: School of Mathematical Sciences, Fudan University, Shanghai 200433, China;
\email{yanjun@fudan.edu.cn}
\and
Kai Zhao:  School of Mathematical Sciences, Fudan University, Shanghai 200433, China;
\email{zhao$_-$kai@fudan.edu.cn}}
\subjclass{37J50; 35F21; 35D40}

\begin{abstract}

Suppose that $H(x,u,p)$ is strictly decreasing in $u$ and satisfies Tonelli conditions in $p$. We show that each viscosity solution  of $H(x,u,u_x)=0$ can be reached by many viscosity solutions of 
$$
w_t+H(x,w,w_x)=0,
$$
in a finite time.

\keywords{Hamilton-Jacobi equations, viscosity solutions, weak KAM theory}
\end{abstract}

%


\section{Introduction}
\setcounter{equation}{0}
\setcounter{footnote}{0}

Let $M$ be a smooth, connected, compact Riemannian manifold without boundary, and $T^*M$ denote the cotangent bundle of $M$. Assume $H:T^*M\times\mathbb{R}\to \mathbb{R}$, $H=H(x,u,p)$, is a $C^3$ function satisfying:
{\bf(H1)} the Hessian $\frac{\partial^2 H}{\partial p^2} (x,u,p)$ is positive definite for all $(x,u,p)\in T^*M\times\R$; {\bf (H2)} for every $(x,u)\in M\times\R$, $H(x,u,p)$ is superlinear in $p$;
{\bf (H3)} there are constants $K_1>0$ and $K_2>0$ such that
	$$
	-K_1 \leqslant \frac{\partial H}{\partial u}(x, u, p)\leqslant-K_2,\quad \forall (x,u,p)\in T^*M\times \R.
	$$
Here, for convenience, we denote $(x,p)\in T^*M$, $u\in\R$, by $(x,u,p)\in T^*M\times\R$.

The notion of viscosity solutions of scalar nonlinear first order Hamilton-Jacobi equations was introduced by Crandall, Evans and Lions \cite{CL}, \cite{CEL}. 
In this paper we aim to understand the long-time behavior of viscosity solutions of 
\begin{equation}\label{CP}\tag{CP}
\begin{cases}
w_t+H(x,w,w_x)=0,\quad (x,t)\in M\times(0,+\infty),\\
w(x,0)=\varphi(x), \quad x\in M,
\end{cases}
\end{equation}
where $\varphi\in C(M,\mathbb{R})$ is the initial data. More precisely, for each viscosity solution $u$ of 
\begin{equation}\label{HJ}\tag{HJ}
H(x,u,u_x)=0,\quad x\in M,
\end{equation}
we are interested in whether there is  $\varphi\in C(M,\mathbb{R})$ different from $u$ such that the unique viscosity solution $w_{\varphi}$ of \eqref{CP} converges to $u$ in a finite time. There have been various literatures dealing with long-time behavior of viscosity solutions of evolutionary Hamilton-Jacobi equations, where convergence results like 
\[
\lim_{t\to +\infty}\big(w_\varphi(x,t)+ct\big)=u(x)\quad \text{for some constant} \ c,
\] 
were proved under various different assumptions, see for instance \cite{Ishii} and the references therein. In contrast, the novelty of this work is the finite-time convergence.

Before stating our main results, we need to introduce some preliminaries.

\medskip
{\bf 1.a.} -- {\it Solution semigroups and  viscosity solutions}. 
Let us recall two  semigroups of operators introduced in \cite{WWY1}.  Define a family of nonlinear operators $\{T^-_t\}_{t\geqslant 0}$ from $C(M,\mathbb{R})$ to itself as follows. For each $\varphi\in C(M,\mathbb{R})$, denote by $(x,t)\mapsto T^-_t\varphi(x)$ the unique continuous function on $ (x,t)\in M\times[0,+\infty)$ such that
\[
T^-_t\varphi(x)=\inf_{\gamma}\left\{\varphi(\gamma(0))+\int_0^tL(\gamma(\tau),T^-_\tau\varphi(\gamma(\tau)),\dot{\gamma}(\tau))d\tau\right\},
\]
where the infimum is taken among the absolutely continuous curves $\gamma:[0,t]\to M$ with $\gamma(t)=x$.  It was also proved  in \cite{WWY1} that $\{T^-_t\}_{t\geqslant 0}$ is a semigroup of operators and the function $(x,t)\mapsto T^-_t\varphi(x)$ is a viscosity solution of \eqref{CP}. Thus, we call $\{T^-_t\}_{t\geqslant 0}$ the backward solution semigroup.

Similarly, one can define another semigroup of operators $\{T^+_t\}_{t\geqslant 0}$, called the forward solution semigroup, by
\begin{equation*}\label{fixufor}
T^+_t\varphi(x)=\sup_{\gamma}\left\{\varphi(\gamma(t))-\int_0^tL(\gamma(\tau),T^+_{t-\tau}\varphi(\gamma(\tau)),\dot{\gamma}(\tau))d\tau\right\},
\end{equation*}
where the supremum is taken among the absolutely continuous curves $\gamma:[0,t]\to M$ with $\gamma(0)=x$.

$T^-_t\varphi(x)$ and $T^+_t\varphi(x)$ can be represented by \cite{WWY1}:
\[
T^-_t\varphi(x)=\inf_{y\in M}h_{y,\varphi(y)}(x,t),\quad  T^+_t\varphi(x)=\sup_{y\in M}h^{y,\varphi(y)}(x,t),\quad (x,t)\in M\times(0,+\infty),
\]
respectively. Here, the continuous functions $h_{\cdot,\cdot}(\cdot,\cdot):M\times\R\times M\times(0,+\infty)\to\R,\ (x_0,u_0,x,t)\mapsto h_{x_0,u_0}(x,t)$ and $h^{\cdot,\cdot}(\cdot,\cdot):M\times\R\times M\times(0,+\infty)\to\R,\ (x_0,u_0,x,t)\mapsto h^{x_0,u_0}(x,t)$ were introduced in \cite{WWY}, called forward and backward implicit action functions respectively.   For any $(x_0,u_0,u,x,t)\in M\times\R\times\R\times M\times(0,+\infty)$, the following relation
\[
h_{x_0,u_0}(x,t)=u\quad \text{if and only if}\quad h_{x,u}(x_0,t)=u_0
\]
holds true. Given $x_0\in M$, $u_1$, $u_2\in\mathbb{R}$, 
if $u_1<u_2$, then $h_{x_0,u_1}(x,t)<h_{x_0,u_2}(x,t)$ and  $h^{x_0,u_1}(x,t)<h^{x_0,u_2}(x,t)$, for all $(x,t)\in M\times (0,+\infty)$. See \cite{WWY,WWY1,WWY2} for more properties of implicit action functions.

\medskip
{\bf 1.b.} -- {\it Weak KAM solutions}.
Following Fathi \cite{Fat-b,Fat-icm}, one can define backward and forward weak KAM solutions of equation \eqref{HJ}, and prove that backward weak KAM solutions and viscosity solutions are the same under assumptions imposed in this paper. Moreover, $u$ is a backward weak KAM solution if and only if $T^-_tu=u$ for all $t\geqslant 0$, and $u$ is a forward weak KAM solution if and only if $T^+_tu=u$ for all $t\geqslant 0$.
See \cite{WWY2,WWY3} for more details.

\medskip
{\bf 1.c.} -- {\it Solvability}. Let $F(x,u,p):=H(x,-u,-p)$. Then $F$ satisfies Tonelli conditions in $p$ and is strictly increasing in $u$. It is a well known fact that 
\begin{equation}\label{F}
F(x,u,Du)=0
\end{equation}
has a unique viscosity solution (or equivalently, backward weak KAM solution). Moreover, \eqref{F} admits at least a forward weak KAM solution. By the relation of weak KAM solutions of \eqref{F} and \eqref{HJ}: $u$ is a backward (resp. forward) weak KAM solution of \eqref{HJ} if and only if $-u$ is a forward  (resp. backward) weak KAM solution of \eqref{F}, it is clear that the set $\mathcal{S}$ of all viscosity solutions (or equivalently, backward weak KAM solutions) of \eqref{HJ} is non-empty, and the forward weak KAM solution of \eqref{HJ} is unique, denoted by $u_+$. Readers can find all the above results in \cite{WWY2}.

\medskip
{\bf 1.d.} -- {\it Decompositions of $C(M,\R)$}.
\begin{align*}
A\ \ &:=\big\{\varphi\in C(M,\R): \min_{x\in M}(\varphi(x)-u_+(x))=0  \big\}, \\
A_+&:=\big\{\varphi\in C(M,\R): \min_{x\in M}(\varphi(x)-u_+(x))>0  \big\}, \\
A_-&:=\big\{\varphi\in C(M,\R): \min_{x\in M}(\varphi(x)-u_+(x))<0 \big\}.
\end{align*}
It is obvious that $C(M,\R)=A \cup A_+ \cup A_-$.

Under assumptions (H1)-(H3), it was proved in \cite[Main Result 2 (1)]{WWY3} that
\[
T_t^- A\subset A,\quad T_t^- A_+\subset A_+, \quad T_t^- A_-\subset A_-,\quad \forall t\geqslant 0.
\]
Moreover, one can deduce that
\begin{itemize}
	\item $\varphi \in A$  if and only if $T_t^- \varphi(x)$ is bounded on $M\times [0,+\infty)$;
	\item $\varphi \in A_+$ if and only if $\lim_{t\to+\infty}T_t^- \varphi(x)=+\infty$ uniformly in $x\in M$;
	\item  $\varphi \in A_-$ if and only if $\lim_{t\to+\infty}T_t^- \varphi(x)=-\infty$ uniformly in $x\in M$. 
\end{itemize}

In view of the above arguments, it is clear that
\[
T^-_tA\subset T^-_sA,\quad \forall s\leqslant t.
\]  
Let 
\[
A_t:=T^-_tA,\quad \forall t\geqslant 0,
\]
and 
\[
A_\infty:=\bigcap_{t\geqslant 0}T^-_tA.
\] 
In view of $\mathcal{S}=\{u\in C(M,\R):T^-_tu=u,\ \forall t\geqslant0\}$, we get that 
$\mathcal{S}\subset A_\infty$.

\medskip
{\bf 1.e.} -- {\it Aubry sets}. For any $\varphi\in A$, define
\[
I_\varphi:=\{x\in M:\varphi(x)=u_+(x)\},
\]
where $u_+$ is the aforementioned unique forward weak KAM solution of \eqref{HJ}.  
Let $u\in\mathcal{S}$. Then $T^-_tu=u$ for all $t\geqslant0$. Thus, one can deduce that $u\in A$. By \cite[Theorem 1.2]{WWY2}, we have
\[
u_+\leqslant u_-\leqslant u\quad \text{everywhere},
\] 
where $u_-$ is the smallest viscosity solution of \eqref{HJ} in the sense of 
\[
u_-(x)=\min_{u\in \mathcal{S}}u(x),\quad \forall x\in M. 
\]
So, it is clear that 
\[
I_u\subset I_{u_-},
\]
where $I_{u_-}$ was called the projected Aubry set in \cite{WWY2}. For any $x\in I_{u_-}$, there is a global calibrated curve passing through it.

\medskip

For any $u\in\mathcal{S}$, let
\[
A_u:=\{\varphi\in A: I_u\subset I_\varphi\}.
\]
It  is  easy to see that $u_- \in A_u$ for any $u\in \mathcal{S}$.

\medskip
{\bf 1.f.} -- {\it Main results}.
Now we are in a position to state our first main result. 
\begin{theorem}\label{th1}
	Let $u\in\mathcal{S}$ and $\varphi\in A_u$. For any $\epsilon>0$, there is $\varphi_\epsilon\in A_u$ with $\|\varphi_\epsilon-\varphi\|_\infty<\epsilon$ such that
	\[
	w_{\varphi_{\epsilon}}(\cdot,t)=u(\cdot),\quad \forall t\geqslant t_0,
	\]
	where $\varphi_\epsilon$ depends on $u$, $\varphi$, $\epsilon$, and $t_0>0$ is a constant depending on  $u$, $\varphi$, $\epsilon$ and $\varphi_\epsilon$.
\end{theorem}
As pointed out in Theorem \ref{th1}, the finite time $t_0$ depends on the initial data. We can also provide the following result where the first reach time  is uniform with respect to initial data.

\medskip

\begin{theorem}\label{th2}
Let $K_2$ be as in {\bf (H3)}.	For any $ \epsilon>0 $ , 
	\begin{align*}
	\mathcal{S}   \subset A_\infty  \subset T_t^-(B_\epsilon(u_+)), \quad \forall t \geqslant \max \bigg \{ \frac{1}{K_2} \ln \frac{ C_1 +1+\|u_+ \|_\infty }{\epsilon},1 \bigg \}, 
	\end{align*}
	where $B_\epsilon (u_+):=\{u\in A,\|u-u_+ \|_\infty<\epsilon\}$, and the constant $C_1>0$ depends only on $u_+$. 
\end{theorem}
This result means that each viscosity solution of \eqref{HJ} can be reached by $T_t^-(\cdot)$ from a neighbourhood of the unique forward weak KAM solution $u_+$ in a uniform finite time $T_0$, where $T_0$ depends only on $u_+$ and the neighbourhood.

Our tools come from some dynamical results on the Aubry-Mather theory and the weak KAM theory for contact Hamiltonian systems \cite{WWY,WWY1,WWY2,WWY3}, where
variational principles \cite{WWY,CCJWY,CCWY} played essential roles.

\medskip
{\bf 1.g.} -- {\it List of symbols}.

\begin{itemize}
	\item $C(M,\R)$: space of continuous functions on $M$
	\item $\|\cdot\|_\infty$: the supremum norm on $C(M,\R)$
	\item $u_+$: the unique forward weak KAM solution of \eqref{HJ} 
	\item $\mathcal{S}$: the set of all viscosity solutions (or equivalently, backward weak KAM soluiotns) of \eqref{HJ}
	\item $\{T_t^\pm\}$: forward and backward solution semigroups associated with $H$
    \item $w_\varphi$: the unique viscosity solution of \eqref{CP}
	
\end{itemize}

The rest of this paper is organized as follows. We prove Theorem \ref{th1} in Section 2. The proof of Theorem \ref{th2} is given in Section 3.

\section{Finite-time convergence}
This section is devoted to the proof of Theorem \ref{th1}.

\begin{proof}[Proof of Theorem \ref{th1}]
Let $\varphi\in A_u$. Since $u\geqslant u_+$ everywhere, by the definition of $I_u$, for any $\epsilon>0$, there are an open neighbourhood $O_\epsilon$ of $I_u$ and $\varphi_\epsilon\in A$, such that: (i) 
$\varphi_\epsilon(x)=u(x)$, $\forall x\in O_\epsilon$; (ii) $\varphi_\epsilon(x)>u_+(x)$, $\forall x\in M\backslash O_\epsilon$; (iii) $\|\varphi_\epsilon-\varphi\|_\infty<\epsilon$. Here, $\|\cdot\|_\infty$ denotes the supremum norm.
Note that $I_u=I_{\varphi_\epsilon}$ and thus $\varphi_\epsilon\in A_u$.

{\bf Step 1}: We aim to show that there is $t_1>0$ such that 
\begin{align}\label{2-1}
T^-_t\varphi_\epsilon(x)=\inf_{y\in O_\epsilon}h_{y,\varphi_\epsilon(y)}(x,t),\quad (x,t)\in M\times[t_1,+\infty).
\end{align}
Let $\delta:=\min_{x\in M\backslash O_\epsilon}(\varphi_\epsilon(x)-u_+(x))$. Then by (ii) $\delta>0$ is well defined. Let $u_\delta:=u_++\delta$. Then $\varphi_\epsilon(x)\geqslant u_\delta(x)$ for all $x\in M\backslash O_\epsilon$. Recall that $\varphi_\epsilon\in A_u\subset A$. Thus, there is a constant $K>0$ depending on $\varphi_\epsilon$ such that 
\[
|T^-_t\varphi_\epsilon(x)|\leqslant K,\quad \forall t\geqslant 0,\ \forall x\in M.
\] 
Since $u_\delta\in A_+$, then
\[
\lim_{t\to+\infty}T^-_tu_\delta(x)=+\infty,\quad \text{uniformly in}\ x\in M.
\]
So, there is $t_1>0$ such that
\[
T^-_tu_\delta(x)\geqslant K+1,\quad \forall t\geqslant t_1,\ \forall x\in M,
\]
where $t_1$ depends on $\epsilon$, $u$, $\varphi$ and $\varphi_\epsilon$. Hence, for any $t\geqslant t_1$, any $x\in M$, we get that
\begin{align}\label{2-2}
\inf_{y\in M\backslash O_\epsilon}h_{y,\varphi_\epsilon(y)}(x,t)\geqslant \inf_{y\in M}h_{y,\varphi_\epsilon(y)}(x,t)\geqslant \inf_{y\in M}h_{y,u_\delta(y)}(x,t)=T^-_tu_\delta(x)\geqslant K+1.	
\end{align}
The second inequality in \eqref{2-2} comes from the monotonicity property of implicit action functions: $v_1\leqslant v_2$ implies $h_{x,v_1}(y,t)\leqslant h_{x,v_2}(y,t)$ for all $(x,y,t)\in M\times M\times (0,+\infty)$. Hence, for any $t\geqslant t_1$, any $x\in M$, by \eqref{2-2}, we have
\[
T^-_t\varphi_\epsilon(x)=
\inf_{y\in M}h_{y,\varphi_\epsilon(y)}(x,t)=\min\big\{\inf_{y\in O_\epsilon}h_{y,\varphi_\epsilon(y)}(x,t),\min_{y\in M\backslash O_\epsilon}h_{y,\varphi_\epsilon(y)}(x,t)\big\}
=\inf_{y\in O_\epsilon}h_{y,\varphi_\epsilon(y)}(x,t).
\]
Thus, \eqref{2-1} holds true.

{\bf Step 2}: Next we show that for above $O_\epsilon$, there is $t_2>0$ such that
\begin{align}\label{2-3}
	u(x)=\inf_{y\in O_\epsilon}h_{y,u(y)}(x,t),\quad \forall t\geqslant t_2,\ \forall x\in M.
\end{align}
Let $\sigma:=\min_{x\in M\backslash O_\epsilon}(u(x)-u_+(x))$. Then by the definition of $I_u$, $\sigma>0$ is well defined. Let $u_\sigma:=u_++\sigma$. Then $u_\sigma\in A_+$ and $u(x)\geqslant u_\sigma (x)$, $\forall x\in M\backslash O_\epsilon$. Thus, 
\[
\lim_{t\to+\infty}T^-_tu_\sigma(x)=+\infty,\quad \text{uniformly in}\ x\in M.
\]
Hence, there is $t_2>0$ such that
\[
T^-_tu_\sigma(x)\geqslant \|u\|_\infty+1,\quad \forall t\geqslant t_2,\ \forall x\in M,
\]
where $t_2>0$ depends on $u$ and $\epsilon$. 
Thus, for any $t\geqslant t_2$ and any $x\in M$, we get 
\begin{align}\label{2-4}
	\inf_{y\in M\backslash O_\epsilon}h_{y,u(y)}(x,t)\geqslant \inf_{y\in M\backslash O_\epsilon}h_{y,u_\sigma(y)}(x,t)\geqslant \inf_{y\in M}h_{y,u_\sigma(y)}(x,t)=T^-_tu_\sigma(x)\geqslant \|u\|_\infty+1.
\end{align}
Since $T^-_tu=u$ for all $t\geqslant 0$, then
\[
 u(x)=T^-_tu(x)=\inf_{y\in M}h_{y,u(y)}(x,t),\quad \forall t\geqslant 0,\ \forall x\in M.
\]
Hence, for any $t\geqslant t_2$ and any $x\in M$, by \eqref{2-4}, we get 
\begin{align}\label{2-5}
	u(x)=\inf_{y\in M}h_{y,u(y)}(x,t)=\min\big\{\inf_{y\in O_\epsilon}h_{y,u(y)}(x,t),\inf_{y\in M\backslash O_\epsilon}h_{y,u(y)}(x,t)\big\}=\inf_{y\in O_\epsilon}h_{y,u(y)}(x,t).
\end{align}

{\bf Step 3}: Let $t_0:=\max\{t_1,t_2\}$. Then by \eqref{2-1} and \eqref{2-5}, we obtain that
\[
T^-_t\varphi_\epsilon(x)=\inf_{y\in O_\epsilon}h_{y,\varphi_\epsilon(y)}(x,t)=\inf_{y\in O_\epsilon}h_{y,u(y)}(x,t)=u(x),
\]
for all $t\geqslant t_0$ and all $x\in M$.

\end{proof}

\section{Uniform finite-time convergence}

\begin{lemma}\label{le1}
For each $x$, $x_0\in M$, $t>0$, $u$, $v\in \R$, there holds	
\[
|h_{x_0,u}(x,t)-h_{x_0,v}(x,t)| \geqslant e^{K_2 t} |u-v|.
\]
\end{lemma}
\begin{proof}
	By the monotonicity property of $h_{x_0,u}(x,t)$ with respect to $u$, if $u< v$ , then $h_{x_0,u}(x,t)< h_{x_0,v}(x,t)$. Let $\gamma_v$ be a minimizer of $h_{x_0,v}(x,t)$  with $\gamma_v(0)=x_0$ and $\gamma_v(t)=x$. Then, for any $s\in [0,t]$,
	\begin{align}\label{eq:mono h}
	h_{x_0,u}(\gamma_v(s),s) \leqslant h_{x_0,v}(\gamma_v(s),s).
	\end{align}
	In terms of the definition of $h_{x_0,v}(x,t)$ and (H3), we have
	\begin{align*}
	&\, h_{x_0,v}(\gamma_v(s),s)-h_{x_0,u}(\gamma_v(s),s)\\
	\geqslant &\, v-u+ \int_0^s L(\gamma_v(\tau),h_{x_0,v}(\gamma_v(\tau),\tau),\dot \gamma_v(\tau ) )-L(\gamma_v(\tau),h_{x_0,u}(\gamma_v(\tau),\tau),\dot \gamma_v(\tau ) ) \ d \tau\\
	\geqslant &\,  v-u + K_2 \int_0^s h_{x_0,v}(\gamma_v(\tau),\tau)-h_{x_0,u}(\gamma_v(\tau),\tau)  \ d \tau
	\end{align*}
	Let  $F(\tau):=h_{x_0,v}(\gamma_v(\tau),\tau)-h_{x_0,u}(\gamma_v(\tau),\tau) $. It follows from \eqref{eq:mono h} that $F(\tau)>0$ for any $\tau \in(0,t]$.Hence, we have
	$$
	F(s) \geqslant v-u +K_2 \int_0^s F(\tau) d \tau,\quad s\in[0,t].
	$$
	It yields $F(t) \geqslant e^{K_2 t} (v-u)$.
	
	Changing the roles of $u$ and $v$, a quite similar argument completes the proof.
\end{proof}

\begin{corollary}\label{cT- -}
	Let $\varphi$,  $\psi \in C(M,\R)$. If $\varphi>\psi$ everywhere, then 
	$$
	T^-_t \varphi(x)- T^-_t \psi(x) \geqslant e^{K_2 t}\min_{y\in M} \{\varphi(y)- \psi(y) \},\quad \forall (x,t)\in M\times(0,+\infty). 
	$$
\end{corollary}
\begin{proof}
	Recall that for each $t>0$ and each $x\in M$, we have 
	\begin{align}\label{2-2500}
	T_t^- \varphi(x)=\inf_{y\in M}h_{y,\varphi(y)} (x,t), \quad  T_t^- \psi(x)=\inf_{y\in M}h_{y,\psi(y)} (x,t).
	\end{align}
	Note that $h_{y,\varphi(y)}(x,t)$  is continuous in $y$. By the compactness of $M$,
	\[
	T_t^-\varphi(x)=h_{y_0,\varphi(y_0)} (x,t)
	\]
	for some $y_0\in M$.
	 By Lemma \ref{le1} and \eqref{2-2500}, for each $t>0$ and each $x\in M$, we have 
	\begin{align*}
	T_t^- \varphi(x)-T_t^- \psi(x)\geqslant &\, h_{y_0,\varphi(y_0)} (x,t)-h_{y_0,\psi(y_0)} (x,t)\geqslant e^{K_2 t}(\varphi(y_0)-\psi(y_0))
	\geqslant e^{K_2 t} \min_{y\in M} \{\varphi(y)- \psi(y) \} .
	\end{align*}
	The proof is complete.
\end{proof}

\begin{lemma}\label{3-3}
	For each $t\geqslant 0$, $T_t^-u_+\geqslant u_+$ everywhere.
\end{lemma}
\begin{proof}
	It is clear that $T_0^-u_+=u_+$. For $t>0$,  we have
	\[
	T^{-}_tu_+(x)=\inf_{y\in M}h_{y,u_+(y)}(x,t),\quad \forall x\in M.
	\]
	Thus, in order to prove $T_t^-u_+\geqslant u_+$ everywhere, it is sufficient to show that for each $y\in M$, $h_{y,u_+(y)}(x,t)\geqslant u_+(x)$ for all $(x,t)\in M\times (0,+\infty)$.
	For any given $(x,t)\in M\times (0,+\infty)$, let $v(y):=h_{y,u_+(y)}(x,t)$ for all $y\in M$. Then $u_+(y)=h^{x,v(y)}(y,t)$. Since
	\[
	u_+(y)=T_t^+u_+(y)=\sup_{z\in M}h^{z,u_+(z)}(y,t),
	\]
	which implies $u_+(y)\geqslant h^{x,u_+(x)}(y,t)$, i.e., $h^{x,v(y)}(y,t)\geqslant h^{x,u_+(x)}(y,t)$. By the monotonicity of backward implicit action functions, we have $v(y)\geqslant u_+(x)$ for all $y\in M$, i.e., $h_{y,u_+(y)}(x,t)\geqslant u_+(x)$ for all $y\in M$.
\end{proof}

\begin{lemma}\label{lem:I compact}
	For any given $t>0$, 
	$$
	C_t:=\sup_{\varphi \in A } \|T^-_t\varphi\|_\infty <+\infty
	$$
	i.e., $A_t$ is bounded  by $C_t$ .
\end{lemma}
\begin{proof} 
Since $\varphi\in A $, then $\varphi \geqslant u_+ $ and thus  $T_t^- \varphi \geqslant T_t^- u_+\geqslant u_+$ everywhere  by Lemma \ref{3-3} . 	
	
On the other hand, recall that $I_\varphi=\{x:\varphi(x)=u_+(x)\}$. Then
\[
	T_t^- \varphi(x) = \inf_{y\in M} h_{y,\varphi(y)}(x,t)\leqslant  \inf_{y\in I_\varphi} h_{y,u_+(y)} (x,t)\leqslant \sup_{y\in I_\varphi} h_{y,u_+(y)} (x,t)\leqslant \sup_{y\in M} h_{y,u_+(y)} (x,t), \quad \forall x\in M.
\]
	Since the function $(x_0,u_0,x,s)\mapsto h_{x_0,u_0}(x,s)$ is continuous on $M\times \R\times M\times (0,+\infty)$, then it is bounded on $M\times [-\|u_+\|_\infty,\|u_+\|_\infty]\times M\times\{t\}$, and  thus	
	$$C_t:=\sup_{\psi \in A_t } \| \psi \|_\infty =\sup_{\varphi \in A } \|T^-_t\varphi\|_\infty <+\infty.
	$$
	 where $C_t$ depends only on $t$ and $u_+$.
\end{proof}

\medskip

\begin{proof}[Proof of Theorem \ref{th2}.]
	For any $\varphi \in A_{1}$ and any $ \epsilon>0 $, define $\varphi_\epsilon \in B_\epsilon (u_+)$ by
	$$
	\varphi_\epsilon(x)=\begin{cases}
	\varphi(x),  & x\in  O_\epsilon, \\
	u_+(x)+\epsilon, \quad & x\in M\backslash O_\epsilon,
	\end{cases}
	$$ 
	where $O_\epsilon:=\{x\in M:\varphi (x)<u_+(x)+\epsilon \} $.
By definition, we get 
	$$
	\varphi(x)|_{M\backslash O_\epsilon } \geqslant \varphi_\epsilon(x)|_{M\backslash O_\epsilon }= u_+ (x)|_{M\backslash O_\epsilon }+\epsilon. 
	$$
	Define $u_\epsilon(x):= u_+(x)+\epsilon $ for all $x\in M$. By  Corollary \ref{cT- -},  we have
	$$
	|T^-_t u_\epsilon(x) - T^-_t u_+(x) | \geqslant e^{K_2 t} \min_{y\in M}\{u_\epsilon(y)-u_+(y)\}=e^{K_2 t}\epsilon,\quad \forall t>0, \ \forall x\in M.
	$$
	Since $T^-_t u_\epsilon>T^-_t u_+$ everywhere, we have
	$$
	T^-_t u_\epsilon(x)  \geqslant e^{K_2 t}\epsilon+T^-_t u_+(x), \quad \forall t>0,\ \forall x\in M.
	$$
Set 
	$$
	T_0:= \max\Bigg\{  \frac{1}{K_2} \ln \frac{ C_1 +1+\|u_+ \|_\infty }{\epsilon} , 1 \Bigg\}, 
	$$
	where $C_1$ is as in Lemma \ref{lem:I compact}. 
	For any $t>0$, by the monotonicity property of implicit action functions and the definition of $\varphi_\epsilon$, we get
	\begin{align}\label{2-9}
	\inf_{y\in M\backslash O_\epsilon} h_{y,\varphi (y)}(x, t) \geqslant \inf_{y\in M\backslash O_\epsilon} h_{y,\varphi_\epsilon (y)}(x, t)=\inf_{y\in M\backslash O\epsilon} h_{y,u_\epsilon (y)}(x, t)
	\geqslant \inf_{y\in M}  h_{y,u_\epsilon (y)}(x, t)
	\end{align}
	For any $t>T_0 >1 $, we have
	\begin{align}\label{2-10}
	\begin{split}
&\,\inf_{y\in M}  h_{y,u_\epsilon (y)}(x, t)	=T_t^-u_\epsilon(x)
	\geqslant  T_t^-u_+(x)+e^{K_2 T_0}\epsilon  \\
	 \geqslant &\,  u_+(x)+C_1 +1+\|u_+ \|_\infty\geqslant  \sup_{\varphi \in A } \|T^-_1 \varphi\|_\infty +1 \\
	 \geqslant &\, \sup_{\varphi \in A } \|T^-_t \varphi\|_\infty +1   =C_t +1.
	\end{split}
	\end{align}
Note that $ T_t^- \varphi$, $T_t^- \varphi_\epsilon \in A_t$. Then by \eqref{2-9} and \eqref{2-10},
	$$
	\inf_{y\in M\backslash O_\epsilon } h_{y,\varphi (y)}(x, t) \geqslant C_t+1\geqslant\| T_t^- \varphi \|_\infty+1 \quad \text{and} \quad \inf_{y\in M\backslash O_\epsilon} h_{y,\varphi_\epsilon (y)}(x, t)\geqslant C_t+1\geqslant \| T_t^- \varphi_\epsilon \|_\infty+1.
	$$
	Hence, we get
	\begin{align*}
	T^-_t \varphi_\epsilon(x)= &\,\inf_{y\in M} h_{y,\varphi_\epsilon(y)}(x,t)=\inf_{y\in O_\epsilon} h_{y,\varphi_\epsilon(y)}(x,t)\\
	=&\,\inf_{y\in O_\epsilon} h_{y,\varphi(y)}(x,t)=\inf_{y\in M} h_{y,\varphi(y)}(x,t)=T^-_t \varphi(x)  , \quad \forall t\geqslant T_0,
	\end{align*}
	 which implies that 
	$$
	A_\infty\subset T_{t}^- A_1\subset T_{t}^-(B_\epsilon(u_+)), \quad \forall t\geqslant T_0.
	$$
\end{proof}
\begin{example} Consider the following Hamiltonian
	\[
	H(x,u,p)=-2u+p^2,\quad x\in\mathbb{S},\  p\in\R,\ u\in\R.
	\]
	Here, $\mathbb{S}:=(-\frac{1}{2},\frac{1}{2}]$ denotes the unit circle. 
	The corresponding ergodic Hamilton-Jacobi equation reads
	\begin{align}\label{2-2600}
	-2u+(u')^2=0,\quad x\in\mathbb{S}.
	\end{align}
	Let $u_1$ be the even 1-periodic extension of $\frac{1}{2}x^2$ in $[0,\frac{1}{2}]$. Then it is clear that $u_1$ is a viscosity solution of \eqref{2-2600}. 	
	Note that $u=0$ is a viscosity solution of  
	\[
	2u+(u')^2=0.
	\]
	In view of the uniqueness of viscosity solutions of the above equation, $u_+=0$ is the unique forward weak KAM solution of \eqref{2-2600}. Thus, $I_{u_1}=\{0\}$.
	
	Let $\varphi$ be the even 1-periodic extension of $\frac{1}{2}x^2+x$ in $[0,\frac{1}{2}]$. Then one can deduce that $\varphi\in A_{u_1}$.
	For any given small $\epsilon>0$, define $\varphi_\epsilon$ as the even 1-periodic extension of
	\begin{align*}
\bar{\varphi}_\epsilon(x)=\left\{\begin{array}{ll}
\frac{1}{2}x^2,\quad & x\in[0,\epsilon],\\[2mm]
\frac{1}{2}\epsilon^2+ (\frac{3}{2}\epsilon +2) (x-\epsilon),\quad & x\in[\epsilon,2\epsilon],\\[2mm]
\frac{1}{2}x^2+x,\quad & x\in[2\epsilon,\frac{1}{2}].
\end{array}\right.
	\end{align*}
	Then by the proof of Theorem \ref{th1}, there is a finite time $t_0>0$ such that for all $t\geqslant t_0$,
	\[
	w_{\varphi_\epsilon}(x,t)=u_1(x), \quad \forall x\in M.
	\]
Moreover,for $\varphi\in \rm{Lip}(M,\R)$ similarly with the proof of Theorem \ref{th2}, one obtains the following estimation of a finite time in Theorem \ref{th1},
$$
t_0:= \frac{1}{K_2} \ln \frac{ M_0 +1+\|u_+ \|_\infty }{f(\epsilon)},  
$$
where $M_0:= \sup_{\psi \in \mathcal{S}} \| \psi \|_{\infty}$ and
$$
   f(\epsilon )= \min_{\rm{dist}(x,I_{u})=\epsilon / \max\{\rm{Lip}(\varphi),\rm{Lip}(u) \}  }\{ u(x)-u_+(x) \}.
$$
In this example \eqref{2-2600} , $\|u_+ \|_\infty=0 $ , $M_0=\frac{1}{8}$, $K_2=4 $ and $f(\epsilon)=\frac{2}{9}\epsilon^2$. Hence $t_0=\frac{1}{2} \ln \frac{9}{4\epsilon}  $ for given small $\epsilon$.

\end{example}

\vskip 1cm

\section*{Acknowledgements} 
Kaizhi Wang is supported by NSFC Grant No. 11771283, 11931016 and by Innovation Program of Shanghai Municipal Education Commission.
Jun Yan is supported by NSFC Grant No.  11631006, 11790273.


\medskip

\begin{thebibliography}{99}\small
\addcontentsline{toc}{section}{References}
\renewcommand{\baselinestretch}{0.0}
\setlength\itemsep{-2pt}



\bibitem{CCWY}  P. Cannarsa, W. Cheng, K. Wang and J. Yan,  Herglotz' generalized variational principle and  contact type Hamilton-Jacobi equations, {\it Trends in Control Theory and Partial Differential Equations}, 39--67.
 Springer INdAM Ser., \textbf{32}, Springer, Cham, 2019.



\bibitem{CCJWY}
P. Cannarsa, W. Cheng, L. Jin, K. Wang and J. Yan, Herglotz' variational principle and Lax-Oleinik evolution,
J. Math. Pures Appl. \textbf{141} (2020), 99--136.




\bibitem{CL}
M. Crandall and P.-L. Lions, Viscosity solutions of Hamilton-Jacobi equations, Trans. Amer. Math. Soc. \textbf{277} (1983), 1--42.



\bibitem{CEL}
M. Crandall, L. Evans, and P.-L. Lions, Some properties of viscosity solutions of Hamilton-Jacobi equations, Trans. Amer. Math. Soc. \textbf{282} (1984), 487--502. 





\bibitem{Fat-b}
A. Fathi, Weak KAM Theorem in Lagrangian Dynamics, preliminary version 10, Lyon,
unpublished (2008).



\bibitem{Fat-icm}
A. Fathi, Weak KAM theory: the connection between Aubry-Mather theory and viscosity solutions of the Hamilton-Jacobi equation. Proceedings of the International Congress of Mathematicians—Seoul 2014. Vol. III, 597–621, Kyung Moon Sa, Seoul, 2014. 



\bibitem{Ishii}
H. Ishii, Asymptotic solutions for large time of Hamilton-Jacobi equations. Proceedings of the International Congress of Mathematicians. Vol. III, 213--227, Eur. Math. Soc., Zurich, 2006.



%







\bibitem{WWY} K. Wang, L. Wang and J. Yan, Implicit variational principle for contact Hamiltonian systems, Nonlinearity \textbf{30} (2017), 492--515.


\bibitem{WWY1} K. Wang, L. Wang and J. Yan,  Variational principle for contact Hamiltonian systems and its applications, J. Math. Pures Appl. \textbf{123} (2019), 167--200.


\bibitem{WWY2} K. Wang, L. Wang and J. Yan, Aubry-Mather theory for contact Hamiltonian systems, Commun. Math. Phys. \textbf{366} (2019), 981--1023.


\bibitem{WWY3}
K. Wang, L. Wang and J. Yan, Weak KAM solutions of Hamilton-Jacobi equations with decreasing dependence on unknown functions, arXiv:1805.04738v2











\end{thebibliography}
\end{document}